\input ./style/arxiv-general.cfg
\documentclass[MSNbibl,number,citesort,seceqn,dvips]{arxbj}
\makeatletter
   \@ifpackageloaded{graphicx}{}{\usepackage{graphicx}}
\makeatother
\usepackage{upgreek}

\aid{0}
\volume{22}
\issue{1}
\pubyear{2016}
\firstpage{376}
\lastpage{395}
\doi{10.3150/14-BEJ662} 
\docsubty{FLA}

\makeatletter
\newcommand{\rrvert}{\vert}
\newcommand{\llvert}{\vert}
\newcommand{\fraca}[2]{{#1}/{#2}}
\newcommand{\frace}[2]{{#1}/{(#2)}}
\newproclaim{defn}{Definition}[section]
\newremark{rmk}[defn]{Remark}
\newtheorem{them}[defn]{Theorem}
\newproclaim{Assumption}[defn]{Assumption}
\newremark{exm}[defn]{Example}
\newtheorem{cor}[defn]{Corollary}
\newtheorem{prop}[defn]{Proposition}
\newtheorem{lma}[defn]{Lemma}

\renewcommand{\epsilon}{\varepsilon}
\newcommand{\R}{\mathbb{R}}
\newcommand{\E}{\mathbb{E}}
\renewcommand{\P}{\mathbb{P}}
\newcommand{\ud}{\,\mathrm{d}}
\makeatother

\begin{document}
\begin{frontmatter}

\title{Integral representation of random variables with respect to
Gaussian processes}
\runtitle{Integral representation}

\begin{aug}
\author{\inits{L.}\fnms{Lauri}~\snm{Viitasaari}\corref{}\ead[label=e1]{lauri.viitasaari@aalto.fi}}
\address{Department of Mathematics and Systems Analysis, Aalto
University School of Science, Helsinki,
P.O. Box 11100, FIN-00076 Aalto, Finland. \printead{e1}}
\end{aug}

\received{\smonth{3} \syear{2014}}
\revised{\smonth{6} \syear{2014}}

%
\begin{abstract}
It was shown in Mishura \textit{et al.}
(\textit{Stochastic Process. Appl.} \textbf{123} (2013) 2353--2369),
that any random variable can be represented as improper pathwise
integral with respect to fractional Brownian motion.
In this paper, we extend this result to cover a wide class of Gaussian
processes. In particular, we consider a wide class of processes that
are H\"{o}lder continuous of order $\alpha>1/2$ and show that only
local properties of the covariance function play role for such results.
\end{abstract}

%
\end{frontmatter}

\section{Introduction}

In stochastic analysis and its applications such as financial
mathematics, it is an interesting question what kind of random
variables one can
replicate with stochastic integrals. In order to answer this question,
first one needs to consider in which sense the stochastic integral
exists. In particular, if the driving process $X$ is not a
semimartingale it is not clear how to define integrals with respect to
$X$ and what kind of integrands can be integrated with the given
definition of the integral.

The motivation for our work originates back to Dudley \cite{d} who
showed that any functional $\xi$ of a standard Brownian motion $W$ can
be replicated as an
It\^o integral $\int_0^1 \Psi(s)\ud W_s$, where $\Psi$ is an adapted
process satisfying $\int_0^1 \Psi^2(s)\ud s < \infty$ a.s. Moreover,
under additional assumption $\int_0^1 \E[\Psi^2(s)]\ud s < \infty$
one can cover only centered random variables with finite variance. On
the other hand, in this case the process $\Psi$ is unique.

Later on Mishura \textit{et al.} \cite{m-s-v} considered the same problem where
standard Brownian motion $W$ was replaced with fractional Brownian
motion (fBm) $B^H$ with Hurst index $H>\frac{1}{2}$. In this case the
authors considered generalised Lebesgue--Stieltjes integrals with
respect to fBm which can be defined, thanks to results of Azmoodeh
\textit{et al.} \cite{a-m-v}, for integrands of form $f(B^H_u)$ where $f$ is a
function of locally bounded variation. As an application of the results
in \cite{m-s-v},\vadjust{\goodbreak} the authors considered financial implications of the
results and gave a negative answer to the problem of zero integral;
does $\int_0^1 \psi(s) \ud B^H_s=0$ imply that $\psi(s)=0$. This
problem was open for fBm for some time, and in addition the result was
known only for Brownian motion.

It is interesting to note that while the stochastic integrals are
defined in different ways, the results for standard Brownian motion and
fBm are quite similar. On the other hand, the key idea to obtain
representation for arbitrary processes with integrals with respect to
some given process is to use idea of ``tracking'': first define a
sequence which obviously converges and then track that sequence. The
simplest way to do this is to define an integrand on a given time
interval which diverges in the limit and then use stopping times. This
idea was first used by Dudley \cite{d} for Brownian motion and then by
Mishura \textit{et al.} \cite{m-s-v} for fBm.

In this article, motivated by these two contributing works, we study
the problem for more general class of Gaussian processes.
In particular, we also consider generalised Lebesgue--Stieltjes
integrals and show that the brilliant construction introduced in \cite{m-s-v} for fBm applies, with small modifications, for more general
Gaussian processes. We also note that the integrals exist also as
forward integrals in the sense of F\"{o}llmer \cite{Follmer}. Our
class of Gaussian processes consists of wide class of processes which
has versions that are H\"{o}lder continuous of order $\alpha>\frac
{1}{2}$. More precisely, our class of processes consist of H\"older
continuous Gaussian processes $X$ which also satisfy several mild extra
conditions given for the corresponding covariance function $R$. In
particular, the class includes many stationary and stationary increment
processes that are H\"older continuous of sufficient order. In order to
obtain such result for general class of Gaussian processes, we show
that for the construction introduced in \cite{m-s-v} the only required
facts are local properties of the corresponding covariance function.
Moreover, we show that the replication can be done in arbitrary small
amount of time which has significant implications to the finance. As
such, this article is a hybrid of discussing review paper and an
original research article; We prove similar results as for fBm and use
the same idea of tracking so the proofs are quite similar with only
minor changes needed and no unnecessary complexity is added. On the
other hand, the results are extended to much wider class of processes,
the needed properties for such results are identified and it is also
shown that the replication can be done in any time interval. We also
discuss applications such as implications to finance and the problem of
zero integral. In particular, the results of this paper indicate that
with pathwise integrals the answer to the problem of zero integral is
usually false.

The rest of the paper is organised as follows. We start Section~\ref{sec:aux} by recalling the findings obtained in \cite{m-s-v} for fBm.
Moreover, we introduce the key properties of fBm under which the
authors in \cite{m-s-v} obtained their results. We end the Section~\ref{sec:aux} by introducing our notation and assumptions. We also
recall basic facts on generalised Lebesgue--Stieltjes integrals and F\"
{o}llmer integrals. In Section~\ref{sec:main}, we introduce and prove
the main results for our general class of processes. We end the paper
by discussion in Section~\ref{sec:app} where we shortly discuss
financial applications, uniqueness of the representation and the
problem of zero integral.

\section{Auxiliary facts}\label{sec:aux}

\subsection*{Key properties for fractional Brownian motion}

In \cite{m-s-v}, the authors proved the following:
\begin{itemize}
\item
For any distribution function $F$ there exists an adapted process $\Phi
$ such that $\int_0^1 \Phi(s)\ud B_s^H$ is well-defined (in the sense
of generalised Lebesgue--Stieltjes integral) and has distribution~$F$.
\item
Any measurable random variable $\xi$ can be represented as an improper
integral, that is, $\xi= \lim_{t\rightarrow1-}\int_0^t \Psi(s)\ud B^H_s$.
\item
A measurable random variable $\xi$ which is an end value of some H\"
{o}lder continuous process can be represented as a proper integral.
\end{itemize}
Our aim is to establish similar results for general class of Gaussian
processes. By studying the paper \cite{m-s-v}, one can see that in a
sense the following facts are the main ingredients for such results:
\begin{enumerate}[3.]
\item
It\^o's formula: for every locally bounded variation function $f$ we have
\[
F\bigl(B_T^H\bigr) = \int_0^T
f\bigl(B_u^H\bigr)\ud B^H_u,
\]
where $F(x) = \int_0^x f(y)\ud y$,
\item
fBm has stationary increments,
\item
a crossing bound at zero: there exists a constant $C$ such that for
every $0<s<t\leq T$ we have
\[
\P\bigl(B_s^H < 0 < B_t^H\bigr)
\leq C(t-s)^H t^{-H},
\]
\item
small ball probability: there exists a constant $C$ such that for every
$T$ and $\epsilon$ we have
\[
\P\Bigl(\sup_{0\leq t\leq T}\bigl |B_t^H\bigr | \leq\epsilon
\Bigr) \leq\exp \bigl(-CT\epsilon^{-\fraca{1}{H}} \bigr)
\]
provided that $\epsilon\leq T^H$.
\end{enumerate}
For our purposes, we have results similar to conditions 1 and 3 for
more general class of processes obtained by Sottinen and Viitasaari
\cite{s-v2}
(see subsection below). The conditions 2 and 4 we replace with
weaker assumptions on the covariance structure of the Gaussian process $X$.

\subsection*{Definitions and auxiliary results}

Throughout the paper, we are restricted on a bounded interval $[0,T]$
which is usually omitted in notation.

\begin{defn}
Let $X$ be a centered Gaussian process. We denote by $R_X(t,s)$,
$W_X(t,s)$, and $V_X(t)$ its covariance, incremental variance and
variance, that is,
\begin{eqnarray*}
R_X(t,s) &=& \E[X_tX_s],
\\
W_X(t,s) &=& \E\bigl[(X_t - X_s)^2
\bigr],
\\
V_X(t) &=& \E\bigl[X_t^2\bigr].
\end{eqnarray*}
We denote by $w^*_X(t)$ the ``worst case'' incremental variance
\[
w^*_X(t) = \sup_{0\le s\le T-t} W_X(s,s+t).
\]
\end{defn}

Let now $\alpha\in (\frac{1}{2},1 )$. We consider the
following class of processes.

\begin{defn}\label{defn:Xalpha}
A centered continuous Gaussian process $X=(X_t)_{t\in[0,T]}$ with
covariance $R_X$ belongs to the \emph{class} $\mathcal{X}^\alpha_T$
if there is a constant $\delta$ such that for every $u\in[T-\delta
,T)$ the process $Y_t = X_{t+u}-X_u$ for $t\in[0,T-u]$ satisfies:
\begin{enumerate}[3.]
\item
$R_Y(s,t)> 0$ for every $s,t>0$,
\item
the ``worst case'' incremental variance satisfies
\[
w^*_Y(t) = \sup_{0\le s\le T-t-u}W_Y(s,s+t)
\leq Ct^{2\alpha},
\]
where $C> 0$,
\item
there exist $c,\hat{\delta}>0$ such that
\[
V_Y(s) \geq cs^{2}
\]
provided $s\leq\hat{\delta}$,
\item
there exists a $\hat{\delta}>0$ such that
\[
\sup_{0< t<2\hat{\delta}}\sup_{\fraca{t}{2}\leq s\leq t}\frac
{R_Y(s,s)}{R_Y(t,s)}<
\infty.
\]
\end{enumerate}
\end{defn}

The class depends also on parameter $\delta$ which will be omitted on
the notation.

Note that the definition is quite technical. However, the conditions
are needed in order to have It\^o formula and crossing bound for
incremental process $Y$ close to time $T$. Moreover, the results for
fBm relies on the fact that $B^H$ has stationary increments. For our
class we simply need certain structure for covariance close to $T$. The
idea on the results is that before some point $t=T-\delta$ we simply
wait and do nothing. Moreover, the following remarks and examples show
that the assumptions are not very restrictive and are satisfied for
many Gaussian processes. For further discussion and details, see \cite{s-v2}
where the class was first introduced such that the covariance of
$X$ itself satisfy properties 1--4.\vadjust{\goodbreak}

\begin{rmk}
\begin{enumerate}[3.]
\item
Note that the first condition means that the increments of the process
are positively correlated close to time $T$. More precisely, we need
\[
R_X(t+u,s+u)+R_X(u,u) > R_X(t+u,u) +
R_X(u,s+u).
\]
In other words, the covariance function should have positive increments
on rectangles.
\item
The second condition implies that $Y$ has version which is H\"{o}lder
continuous of any order $a<\alpha$. For the rest of the paper, we
assume that this version is chosen.
\item
A special subclass of $\mathcal{X}^\alpha_T$ are processes with
stationary increments. In this case, we have
\begin{eqnarray*}
R_Y(t,s) &=& R_X(t,s)=\tfrac{1}{2}
\bigl[V(t)+V(s)-V(t-s) \bigr],
\\
W_Y(t,s) &=& W_X(t,s)=V_X(t-s),
\\
w^*_Y(t) &=& w^*_X(t)=V_X(t).
\end{eqnarray*}
Especially, stationary increment processes with $W_X(t,s)\sim
|t-s|^{2\alpha}$ at zero with $\alpha>\frac{1}{2}$ belong to
$\mathcal{X}^\alpha_T$ for every $T$. In particular, fBm with Hurst
index $H>\frac{1}{2}$ belongs to $\mathcal{X}^\alpha_T$.
\item
Another special subclass of $\mathcal{X}^\alpha_T$ are stationary
processes. In this case, we have
\begin{eqnarray*}
R_X(t,s) &=& r(t-s),
\\
W_X(t,s) &=& 2 \bigl[r(0)-r(t-s) \bigr],
\\
V_X(t) &=& r(0),
\\
w^*_X(t) &=& 2 \bigl[r(0)-r(t) \bigr]
\end{eqnarray*}
and
\begin{eqnarray*}
R_Y(t,s) &=& r(t-s)+r(0)-r(t)-r(s),
\\
W_Y(t,s) &=& W_X(t,s),
\\
V_Y(t) &=& W_X(t+u,u) = w^*_X(t),
\\
w^*_Y(t) &=& w^*_X(t).
\end{eqnarray*}
Consequently, for a stationary process $X$ with covariance function
$r(t)$ we have $X\in\mathcal{X}^\alpha_T$ if
$r(t)$ satisfies
\begin{eqnarray*}
r(t-s)+r(0) &>& r(t) + r(s),
\\
ct^2&\leq& r(0)-r(t) \leq Ct^{2\alpha}
\end{eqnarray*}
and
\[
\sup_{0< t<2\hat{\delta}}\sup_{\fraca{t}{2}\leq s\leq t} \frac
{r(0)-r(s)}{r(t-s)+r(0)-r(t)-r(s)} <
\infty.
\]
Especially, processes with strictly decreasing covariance at zero
satisfy assumptions 1 and~4. In particular, stationary
processes with strictly decreasing covariance and $W_X(t,s)\sim
|t-s|^{2\alpha}$ at zero with $\alpha>\frac{1}{2}$ belongs to
$\mathcal{X}^\alpha_T$ for every $T$. As an example, the process $X$
with covariance function
\[
r(t) = \exp \bigl(-|t|^{2\alpha} \bigr)
\]
with $\frac{1}{2}<\alpha<1$
belongs to $\mathcal{X}^\alpha_T$. We will use this process as a
motivating example throughout the paper, and we will denote this
process by $\tilde{X}$.
\end{enumerate}
\end{rmk}

The following statement derived in Sottinen and Viitasaari \cite{s-v2}
is one of the main ingredients for our study.

\begin{them}
\label{thm:ito}
Let $X\in\mathcal{X}^{\alpha}_T$ with $\alpha>\frac{1}{2}$ and let
$f$ be a function of locally bounded variation. Set $F(x) =\int_0^x
f(y)\ud y$. Then
%
\begin{equation}
\label{ito} F(X_T - X_u) = \int_u^T
f(X_s - X_u) \ud X_s
\end{equation}
provided $u\in[T-\delta,T)$, where the integral can be understood as
a generalised Lebesgue--Stieltes integral or as F\"{o}llmer integral.
\end{them}

\begin{rmk}
In the original paper \cite{s-v2}, the authors considered only convex
functions. However, by examining the proof it is evident that the
result holds also for functions of locally bounded variation.
\end{rmk}

Furthermore, we make the following assumption for small ball
probabilities. The examples are discussed in the next subsection.

\begin{Assumption}\label{assu:smallball}
There exist constants $C,\delta>0$ such that for every $s,t\in
[T-\delta,T]$ with $t=s+\Delta$ it holds
%
\begin{equation}
\label{smallball} \P\Bigl(\sup_{s\leq u \leq t} |X_u-
X_s| \leq\epsilon\Bigr) \leq \exp \bigl(-C\Delta\epsilon^{-\fraca{1}{\alpha}}
\bigr)
\end{equation}
provided that $\epsilon\leq\Delta^\alpha$.
\end{Assumption}

\subsection*{Which processes satisfy the Assumption \texorpdfstring{\protect\ref{assu:smallball}}{2.6}?}

In this subsection, we briefly review what kind of processes $X\in
\mathcal{X}^\alpha_T$ satisfy the Assumption~\ref{assu:smallball}.
In general, the small ball probabilities are an interesting subject of
study and a survey on small ball probabilities is given by Li and Shiao
\cite{l-s} where also the following theorem can be found.

\begin{them}
Let $\{X_t,t\in[0,1\}$ be a centered Gaussian process with $X_0=0$.
Assume that there is a function $\sigma^2(h)$ such that
\[
\forall0\leq s,t\leq1,\qquad \E(X_s-X_t)^2
\leq\sigma^2\bigl(|t-s|\bigr),
\]
and that there are $0<c_1\leq c_2<1$ such that $c_1\sigma(2h\wedge
1)\leq\sigma(h) \leq c_2 \sigma(2h\wedge1)$ for $0< h<1$. Then
there exists $K>0$ depending only on $c_1$ and $c_2$ such that
\[
\P \Bigl(\sup_{0\leq t\leq1}|X_t| \leq\sigma(\epsilon)
\Bigr) \geq\exp \biggl(-\frac{K}{\epsilon} \biggr).
\]
\end{them}

\begin{exm}
It is straightforward that fBm satisfies the assumptions for any $H\in(0,1)$.
\end{exm}

As a direct consequence, we obtain the following statement.

\begin{cor}
Let $X\in\mathcal{X}^\alpha_T$. Then for every $t\in[0,T]$ there
exist $\Delta>0$ and $K>0$ such that
\[
\P \Bigl(\sup_{s\leq u\leq t}|X_u-X_s| \leq
\epsilon \Bigr) \geq \exp \bigl(-K\Delta\epsilon^{-\fraca{1}{\alpha}} \bigr),
\]
provided that $|t-s|\leq\Delta$.
\end{cor}

According to this corollary the bound given in
Assumption~\ref{assu:smallball} is the best possible in terms of $\Delta$ and
$\epsilon$. The upper bound is more difficult to obtain. Moreover, it
is pointed out in \cite{l-s} that the incremental variance is not an
appropriate tool for the upper bound. However, in many cases of
interest we can have the required upper bound. In particular, many
cases of interest have stationary increments or
are stationary processes. For processes with stationary increments, the
following theorem can be used to study the upper bound. For the proof,
we refer to \cite{k-l-s} where a slightly more general setup was considered.

\begin{them}
\label{thm:stat_inc_smallball}
Assume that the centered process $X$ has stationary increments and the
incremental variance $W(t,s) = W(0,t-s)$ satisfies:
\begin{enumerate}[2.]
\item
There exists $\theta\in(0,4)$ such that for every $x\in
[0,\frac{1}{2} ]$ we have
\[
W(0,2x) \leq\theta W(0,x).
\]
\item
For every $0<x<1$ and $2\leq j\leq\frac{1}{x}-2$, we have
%
\begin{eqnarray}
\label{hassu-konditio} %
&&6W(0,jx) + W\bigl(0,(j+2)x\bigr)+W\bigl(0,(j-2)x
\bigr)
\nonumber
\\[-8pt]
\\[-8pt]
&&\quad\geq4W\bigl(0,(j+1)x\bigr) + 4W\bigl(0,(j-1)x\bigr).
\nonumber
\end{eqnarray}
\end{enumerate}
Then there exists a constant $K>0$ such that for every $\epsilon\in
(0,1)$ we have
\[
\P \Bigl(\sup_{0\leq t\leq1}|X_t-X_0| \leq
\sqrt{W(0,\epsilon )} \Bigr) \leq\exp \biggl(-\frac{K}{\epsilon} \biggr).
\]
\end{them}

\begin{rmk}
In the original theorem, it was stated that instead of
(\ref{hassu-konditio}) it is also sufficient that the
incremental variance $W(t,s)$ is concave. Note that in our case usually
$W(0,t)\sim t^{2\alpha}$ with $\alpha>\frac{1}{2}$. Hence, $W(t,s)$
cannot be concave.
\end{rmk}

\begin{rmk}
We remark that the result holds also for stationary Gaussian processes.
\end{rmk}

\begin{cor}
Assume that $X\in\mathcal{X}^\alpha_T$ has stationary increments or
is stationary such that $W(0,t)\sim t^{2\alpha}$. Then
Assumption~\ref{assu:smallball} is satisfied.
\end{cor}

\begin{pf}
It is straightforward to see that a function $W(0,x)=x^{2\alpha}$
satisfies (\ref{hassu-konditio})
provided \mbox{$\alpha> \frac{1}{2}$}. It remains to note that with $\delta
$ small enough, we have $W(0,t-s) \sim C|t-s|^{2\alpha}$ provided
\mbox{$|t-s|\leq\Delta$}.
\end{pf}
%
\begin{exm}
As special examples we note that fBm $B^H$ with $H>\frac{1}{2}$ and
the process $\tilde{X}$ satisfy the Assumption~\ref{assu:smallball}.
\end{exm}

For general processes, $X\in\mathcal{X}^\alpha_T$ it is not clear
when Assumption~\ref{assu:smallball} is satisfied.
In principle, one can derive similar result as
Theorem~\ref{thm:stat_inc_smallball} under similar conditions.
However, in this case the incremental variance function $W(t+s,s)$
depends also on the starting point $s$.
Consequently, one needs to check the condition when $s$ is close to
$T$. Hence in this case, the structure of the covariance function is
more important.

\subsection*{Pathwise integrals}

In this section, we briefly introduce two kinds of pathwise integrals.

\subsubsection*{Generalized Lebesgue--Stieltjes Integral}

The generalized Lebesgue--Stieltjes integral is based on fractional
integration and fractional Besov spaces. For details on these topics,
we refer to \cite{s-k-m} and \cite{n-r}.

Recall first the definitions for fractional Besov norms and
Lebesgue--Liouville fractional integrals and derivatives.

\begin{defn}
Fix $ 0 <\beta< 1 $.
\begin{enumerate}[2.]
\item
The \emph{fractional Besov space} $W^{\beta}_1 = W^{\beta}_1
([0,T])$ is the space of real-valued measurable functions $ f \dvtx [0,T]
\to\mathbb{R}$ such that
\[
{\Vert f \Vert}_{1,\beta} = \sup_{0 \le s < t \le T} \biggl(
\frac
{|f(t) - f(s)|}{(t-s)^\beta} + \int_{s}^{t} \frac{|f(u) - f(s)
|}{(u-s)^ {1+\beta}}
\ud u \biggr) < \infty.
\]
\item
The \emph{fractional Besov space} $W^{\beta}_2 = W^{\beta}_2
([0,T])$ is the space of real-valued measurable functions $ f \dvtx [0,T]
\to\mathbb{R}$ such that
\[
{\Vert f \Vert}_{2,\beta} = \int_{0}^{T}
\frac{|f(s)|}{s^ \beta} \ud s + \int_{0}^{T}\int
_{0}^{s} \frac{|f(u) - f(s) |}{(u-s)^
{1+\beta}} \ud u \ud s < \infty.
\]
\end{enumerate}
\end{defn}

In this paper, we study the norm ${\Vert f \Vert}_{2,\beta}$ on
different intervals $[0,t]$. Hence we use short notation ${\Vert f
\Vert}_{t,\beta}$.

\begin{rmk}\label{r:rmk1}
Let $C^{\alpha}=C^{\alpha}([0,T])$ denote the space of H\"{o}lder
continuous functions of order $\alpha$ on $[0,T]$ and let $ 0<
\epsilon< \beta\wedge(1- \beta)$. Then
\[
C^{\beta+ \epsilon} \subset W^{\beta}_{1} \subset
C^{\beta-
\epsilon} \quad\mbox{and}\quad C^{\beta+ \epsilon} \subset
W^{\beta}_{2}.
\]
\end{rmk}

\begin{defn}
Let $t\in[0,T]$. The \emph{Riemann--Liouville fractional integrals}
$I^\beta_{0+}$ and $I^\beta_{t-}$ of order $\beta> 0$ on $[0,T]$ are
\begin{eqnarray*}
\bigl(I^\beta_{0+} f\bigr) (s) &=& \frac{1}{\Gamma(\beta)} \int
_0^s f(u) (s-u)^{\beta-1} \ud u,
\\
\bigl(I^\beta_{t-} f\bigr) (s) &=& \frac{\mathrm{e}^{\mathrm{i}\uppi\beta}}{\Gamma(\beta)} \int
_s^t f(u) (u-s)^{\beta-1} \ud u,
\end{eqnarray*}
where $\Gamma$ is the Gamma-function. The \emph{Riemann--Liouville
fractional derivatives} $D^{\beta}_{0+}$ and $D^{\beta}_{t-}$ are the
left-inverses of the corresponding integrals $I^\beta_{0+}$ and
$I^\beta_{t-}$. They can be also define via the \emph{Weyl
representation} as
\begin{eqnarray*}
\bigl(D^{\beta}_{0+} f\bigr) (s) &=& \frac{1}{\Gamma(1-\beta)} \biggl(
\frac{f(s)}{s^\beta} + \beta \int_{0}^{s}
\frac{f(s) - f(u)}{(s-u)^{\beta+ 1}} \ud u \biggr),
\\
\bigl(D^{\beta}_{t-} f\bigr) (s) &=& \frac{\mathrm{e}^{\mathrm{i}\uppi\beta}}{\Gamma(1-\beta)} \biggl(
\frac
{f(s)}{(t-s)^\beta} + \beta\int_{s}^{t}
\frac{f(s) -
f(u)}{(u-s)^{\beta+ 1}} \ud u \biggr)
\end{eqnarray*}
if $f\in I^\beta_{0+}(L^1)$ or $f\in I^\beta_{t-}(L^1)$, respectively.
\end{defn}

Denote $g_{t-}(s) = g(s)-g(t-)$.

The generalized Lebesgue--Stieltjes integral is defined in terms of
fractional derivative operators according to the next proposition.

\begin{prop}[(\cite{n-r})]\label{pr:n-r}
Let $0<\beta<1$ and let $f \in W^{\beta}_2$ and $g\in W^{1-\beta
}_1$. Then for any $t \in(0,T]$ the \emph{generalized
Lebesgue--Stieltjes integral} exists as the following Lebesgue integral
\[
\int_0^t f(s) \ud g(s) = \int
_{0}^{t} \bigl(D^{\beta}_{0+}
f_{0+}\bigr) (s) \bigl(D^{1- \beta}_{t-} g_{t-}
\bigr) (s) \ud s
\]
and is independent of $\beta$.
\end{prop}

We will use the following estimate to prove the existence of F\"{o}llmer integrals.

\begin{them}[(\cite{n-r})]\label{t:n-r}
Let $ f \in W^{\beta}_2$ and $ g \in W^{1- \beta}_1$. Then we have
the bound
\[
\biggl\llvert \int_{0}^t f(s) \ud g(s) \biggr
\rrvert \le\sup_{0\le s < t \le T} \bigl | D^{1 - \beta}_{t-}
g_{t-}(s)\bigr  | {\Vert f \Vert}_{2,\beta}.
\]
\end{them}

\subsubsection*{F\"{o}llmer integral}

We also recall the definition of a forward-type Riemann--Stieltjes
integral due to F\"{o}llmer \cite{Follmer} (for English translation,
see \cite{Sondermann}).

\begin{defn}\label{defn:follmer-integral}
Let $(\pi_n)_{n=1}^{\infty}$ be a sequence of partitions
$\pi_n=\{0=t_0^n<\cdots<t_{k(n)}^n=T\}$ such that $|\pi_n|=\max_{j=1,\ldots,k(n)}|t_j^n-t_{j-1}^n|\rightarrow0$ as $n\to\infty$.
Let $X$ be a continuous process. The \emph{F\"{o}llmer integral along
the sequence} $(\pi_n)_{n=1}^{\infty}$ of $Y$ with respect to $X$ is
defined as
\[
\int_0^t Y_u \ud
X_u = \lim_{n\rightarrow\infty} \sum
_{t_j^n\in
\pi_n \cap(0,t]}Y_{t_{j-1}^n}(X_{t_j^n}-X_{t_{j-1}^n}),
\]
if the limit exists a.s.
\end{defn}

The F\"{o}llmer integral is a natural choice for applications such as
finance. However, usually it is difficult to prove the existence of the
F\"{o}llmer integral. For instance, for finite quadratic variation
processes the existence of the integral is a consequence of the It\^o's
formula. On the other hand, generalised Lebesgue--Stieltjes integrals
provides a tool to obtain the existence of F\"{o}llmer integral. For
instance, in \cite{s-v2} the authors proved first the existence of a
generalised Lebesgue--Stieltjes integral and then obtained the existence
of F\"{o}llmer integral by applying Theorem~\ref{t:n-r}.

\section{Main results}\label{sec:main}

We begin with the following technical lemma which gives the diverging
integrand. In our case, it can be defined similarly as for fBm. Hence,
we simply present the key points of the proof.

\begin{lma}\label{lma:aux}
Let $X\in\mathcal{X}^\alpha_T$ such that
Assumption~\ref{assu:smallball} is satisfied.
Then one can construct $\mathbb{F}$-adapted process $\phi_T$ on
$[0,T]$ such that the integral
\[
\int_0^s \phi_T(s)\ud
X_s
\]
exists for every $s<T$ and
%
\begin{equation}
\label{rep:aux-lemma} \lim_{s\rightarrow T-} \int_0^s
\phi_T(s)\ud X_s = \infty
\end{equation}
a.s.\vadjust{\goodbreak}
\end{lma}

\begin{pf}
Fix numbers $\gamma\in (1,\frac{1}{\alpha} )$ and $\eta
\in (0,\frac{1}{\gamma\alpha}-1 )$. Furthermore,
set $t_0=0$ and $t_n= \sum_{k=1}^n \Delta_k$, $n\geq1$ where $\Delta
_n = \frac{Tn^{-\gamma}}{\sum_{k=1}^\infty k^{-\gamma}}$,
and define a function $f_{\eta}(x)=(1+\eta)|x|^{\eta}\operatorname{sign}(x)$.
Note that we can assume without loss of generality that
conditions of Definition~\ref{defn:Xalpha} hold in the whole interval. Otherwise set
$t_1=T-\delta$ and start after $t_1$. Finally, we set
\[
\tau_n = \min \bigl\{t\geq t_{n-1} \dvt |X_t -
X_{t_{n-1}}| \geq n^{-\frace{1}{1+\eta}} \bigr\} \wedge t_n
\]
and
\[
\phi_T(s) = \sum_{n=1}^\infty
f_\eta(X_s - X_{t_{n-1}})\mathbf {1}_{[t_{n-1},\tau_n)}(s).
\]
In order to complete the proof, we have to show that $\Vert \phi
_T\Vert _{s,\beta}<\infty$ a.s. for every $s<T$ and that
(\ref{rep:aux-lemma}) holds. The fact that $\Vert \phi_T\Vert _{s,\beta
}<\infty$ can be proved similarly as for fBm case in \cite{m-s-v}
together with Theorem~\ref{thm:ito}. Hence, it remains to show that
(\ref{rep:aux-lemma}) holds.

First by Theorem~\ref{thm:ito}, we get that for every
$s\in[t_{n-1},t_n)$
\[
\int_0^s \phi_T(u)\ud
X_u = \sum_{k=1}^{n-1}|X_{\tau
_k}-X_{t_{k-1}}|^{1+\eta}
+ |X_{s\wedge\tau_n} - X_{t_{n-1}}|^{1+\eta}.
\]
Now, as in the case of fBm, it is enough to show that only finite
numbers of events $A_n$ happen where $A_n$ is defined by
\[
A_n = \Bigl\{\sup_{t_{n-1}\leq t\leq t_n}|X_t -
X_{t_{n-1}}|< n^{-\frace{1}{\eta+1}} \Bigr\}.
\]
But now, by Assumption~\ref{assu:smallball}, we have
\[
\P(A_n) \leq \mathrm{e}^{-Cn^{-\gamma+ \frace{1}{\alpha(\eta+1)}}}
\]
for $n$ large enough. Noting our choices of $\gamma$ and $\eta$ we
obtain $\sum_{n\geq1}\P(A_n) < \infty$, and thus the result follows
from Borel--Cantelli lemma.
\end{pf}

\begin{rmk}
Same result can be obtained for integrals over any interval
$[s,t]\subset[T-\delta,T]$.
\end{rmk}

\begin{rmk}
It was remarked in paper by Mishura \textit{et al.} \cite{m-s-v} that for fBm
it is easy to see
that $\Vert \phi_T\Vert _{t,\beta}<\infty$ even for random times $t<T$. This
is indeed natural, since the It\^o's formula (\ref{ito}) holds also
for any bounded random time $\tau$ (see \cite{s-v2} for details).
\end{rmk}

\begin{rmk}
It was shown in \cite{a-m-v} that for fBm one can approximate the
integral of It\^o's
formula (\ref{ito}) with Riemann--Stieltjes sums along uniform
partition, i.e. the integral exists also as F\"{o}llmer integral.
Moreover, it was pointed out in \cite{s-v2} that this is true for more
general processes $X\in\mathcal{X}^\alpha_T$ and any partition.
Hence for any $n$, the integral
\[
\int_{t_{n-1}}^{t_n} f_{\eta}(X_s -
X_{t_{n-1}})\ud X_s
\]
exists also as F\"{o}llmer integral. Now by noting that
$\phi_T(s)$ is defined as a linear combination of functions of this
form it is evident that the integral
\[
\int_0^t \phi_T(s)\ud
X_s
\]
exists also as F\"{o}llmer integral for every $t<T$. The same
conclusion holds true also for other results presented in this paper.
\end{rmk}

As a direct corollaries, we obtain that integral with respect to $X_t$
can have any distribution and that any measurable random variable can
be represented as an improper integral; same results as for fBm. For
the sake of completeness, we present the results.

\begin{cor}
\label{cor:distribution_rep}
For any cdf $F$ one can construct adapted process $\psi_T(s)$ such
that $\int_0^T \psi_T(s) \ud X_s$ has distribution $F$.
\end{cor}

\begin{pf}
The proof follows same arguments as for fBm in \cite{m-s-v} except
that since we do not know how the process $X$ behaves before some time
close to $T$, we have to choose some point $v<T$ such that $X_v$ has
non-vanishing variance. The rest follows with same arguments with
obvious changes.\vadjust{\goodbreak}
\end{pf}
%
\begin{rmk}
Note that the result remains true if replace the process $X$ with $Y =
h(X)$, where $h$ is
strictly monotone $C^1$ function. In this case the integrals of form
\[
\int_0^T \psi_T(s) \ud
Y_s
\]
are well defined by results in \cite{s-v2}. We remark that the result
is still valid even if the function $h$ is uniformly bounded.
\end{rmk}

\begin{them}
\label{thm:arb_rv-rep}
Let $(\Omega,\mathcal{F},\P)$ be a complete probability space with
left-continuous filtration $\mathbb{F}=\{\mathcal{F}_t\}_{t\in
[0,T]}$ and let $X\in\mathcal{X}^\alpha_T$ such that
Assumption~\ref{assu:smallball} is satisfied.
Then for any $\mathcal{F}_T$ measurable random variable $\xi$ one can
construct
$\mathbb{F}$-adapted process $\Psi_T$ on $[0,T]$ such that the integral
\[
\int_0^s \Psi_T(s)\ud
X_s
\]
exists for every $s<T$ and
\[
\lim_{s\rightarrow T-} \int_0^s
\Psi_T(s)\ud X_s = \xi
\]
a.s.
\end{them}
\begin{pf}
As in the proof of Lemma~\ref{lma:aux}, we can assume that assumptions
of Definition~\ref{defn:Xalpha} are satisfied for the whole interval.
Put first $Y_t = \tan\E[\operatorname{arctan}\xi|\mathcal{F}_t]$. Now
$Y_t$ is adapted, and we have
$Y_t \rightarrow\xi$ as $t\rightarrow T-$ a.s. by martingale
convergence theorem and left continuity of $\mathbb{F}$. Next for a
sequence $t_n$ increasing to $T$, set
$\delta_n = Y_{t_n} - Y_{t_{n-1}}$ and $\tau_n= \inf \{t \geq
t_n \dvt  Z_t^n = |\delta_n| \}$, where $Z_t^n = \int_{t_n}^t \phi
_{t_{n+1}}(s)\ud X_s$, and $\phi_{t_{n+1}}(s)$ is the process
constructed in Lemma~\ref{lma:aux} such that $Z_t^n \rightarrow\infty
$ as $t\rightarrow t_{n+1}$. By setting
\[
\Psi_T(s) = \sum_{n\geq1}
\phi_{t_{n+1}}(s)\mathbf{1}_{[t_n,\tau
_n]}(s)\operatorname{sign}(
\delta_n)
\]
we can repeat the arguments in \cite{m-s-v} to conclude that
\[
V_t := \lim_{t\rightarrow T-}\int_0^t
\Psi_T(s)\ud X_s = \xi.
\]
\upqed\end{pf}
%
\begin{rmk}
Consider an arbitrary $\mathbb{F}$ measurable process $Y_t$. If for
every $t\in(0,T]$ we have $X\in\mathcal{X}^\alpha_t$, then by
Theorem~\ref{thm:arb_rv-rep} we have that for
every $t$ there is a process $\Psi_t(u)$ such that the process
\[
V_t := \lim_{s\rightarrow t-}\int_0^s
\Psi_t(u)\ud X_u
\]
is a version of $Y_t$.
\end{rmk}

For the proof of our main theorem we also need a bound for the
probability that a Gaussian process $X$ crosses a zero level. The bound
is a consequence of the following more general result proved in \cite{s-v2}.

\begin{lma}
\label{lma:crossing2}
Let $X$ be a centered Gaussian process with strictly positive and
bounded covariance function $R$, $0 < s < t \le T$
and $ a \in\R$. Then there exists a universal constant $C=C(T)$ such that
\begin{eqnarray*}
&&{\P ( X_s < a < X_t )}
\\
&&\quad \le C \frac{\sqrt{W(t,s)}}{\sqrt{V(s)}} \biggl[1+\frac
{R(s,s)}{R(t,s)}+ \frac{|a|\mathrm{e}^{-\frace{a^2}{2V^*}}}{\sqrt{V(s)}}
\max \biggl(1,\frac{R(s,s)}{R(t,s)} \biggr) \biggr],
\end{eqnarray*}
where
\[
V^* = \sup_{s\leq T}V(s).
\]
\end{lma}

\begin{cor}
\label{lma:crossing}
Let $X$ be a centered Gaussian process with positive and bounded
covariance function $R(s,t)$, and let $0<s\leq t\leq T$ be fixed. Then
there exists a constant $C=C(T)$ such that
\[
\P(X_s < 0 < X_t) \leq C \sqrt{\frac{W(t,s)}{V(s)}}
\biggl[1+\frac
{R(s,s)}{R(t,s)} \biggr].
\]
\end{cor}

In \cite{m-s-v} the authors also studied when a random variable $\xi$
can be viewed as a proper integral, that is,
\[
\xi= \int_0^1 \Psi(s)\ud B^H_s
\]
for some process $\Psi(s)$. As a result it was shown in \cite{m-s-v}
that this is true if $\xi$ can be viewed as an endpoint of some
stochastic process which is H\"{o}lder continuous of some order $a>0$.
Moreover, under assumption that $\Psi$ is continuous the authors also
proved that the conditions are necessary. As the proof is based on
similar arguments as the proofs of previous theorems, it is not a
surprise that we can derive similar results for our general class of
processes. However, in our general case we have to modify the proof
accordingly by choosing parameters differently. Consequently, we can
only cover processes $\xi$ which are H\"{o}lder continuous of order
$a>1-\alpha$. For extensions, see Remark~\ref{rmk:H} below.

\begin{them}
\label{thm:H_rv-rep}
Let $X\in\mathcal{X}^\alpha_T$ such that Assumption~\ref{assu:smallball} is satisfied, and
let $\xi$ be $\mathcal{F}_T$ measurable random variable. If there
exists a H\"{o}lder continuous process $Z_s$ of order
$a>1-\alpha$ such that $Z_T = \xi$, then one can construct $\mathbb
{F}$-adapted process $\Psi_T$ on $[0,T]$ such that
the integral
\[
\int_0^T \Psi_T(s)\ud
X_s
\]
exists and
\[
\int_0^T \Psi_T(s)\ud
X_s = \xi
\]
a.s.
\end{them}

As in the proof of Lemma~\ref{lma:aux} and without loss of generality,
we assume that conditions of Definition~\ref{defn:Xalpha} are
satisfied for the whole interval. Otherwise we simply choose $t_1$
large enough such that we are close to $T$.
\begin{pf*}{Proof of Theorem~\ref{thm:H_rv-rep}}
Without loss of generality, we can assume $a<\alpha$. Let   $\beta
\in (1-\alpha,a\wedge\frac{1}{2} )$ and fix $\gamma>
\frac{1}{a-\beta}\vee1$. We
put $\Delta_n = \frac{Tn^{-\gamma}}{\sum_{k=1}^{\infty}k^{-\gamma
}}$ and set $t_0=0$, $t_n = \sum_{k=1}^{n-1}\Delta_k$, $n\geq2$.
Note that with our choice of $\gamma$ and $\beta$ we have $\gamma
(\alpha-\beta)-1> \gamma(\alpha-a)$. Hence, we can choose some
$\kappa\in(\gamma(\alpha- a), \gamma(\alpha-\beta)-1)$. Next, we
proceed as for fBm case and
divide the proof into three steps:
\begin{enumerate}[3.]
\item Set $\Psi_T(t) = 0$ on interval $[t_0,t_1]$. To proceed the
construction is done recursively on intervals $(t_n,t_{n+1}]$ and the
construction is divided into two steps depending on whether we have
$Y_{t_{n-1}} = Z_{t_{n-2}}$ (Case A) or $Y_{t_{n-1}} \neq Z_{t_{n-2}}$
(Case B).
For the sake of completeness and clearness, we present the steps.

Put $Y_t = \int_0^t \Psi_T(s)\ud X_s$ and assume that $\Psi_T(s)$ is
constructed on $[0,t_{n-1}]$ for some $n\geq2$. If we have Case A,
then we set
\[
\tau_n = \inf \bigl\{t\geq t_{n-1} \dvt
n^{\kappa}|X_t- X_{t_{n-1}}|= |Z_{t_{n-1}}-Z_{t_{n-2}}|
\bigr\} \wedge t_n
\]
and for $s\in[t_{n-1},t_n)$,
\[
\Psi_T(s) = n^{\kappa}\operatorname{sign}(X_s -
X_{t_{n-1}})\operatorname{sign}(Z_{t_{n-1}}-Z_{t_{n-2}})
\mathbf{1}_{[0,\tau_n]}(s).
\]
Now if $\tau_n < t_n$, we obtain by It\^o's formula (\ref{ito}) that
\[
Y_{t_n} = Z_{t_{n-1}}.
\]
Assume next that we have Case B. Then we proceed as in
Theorem~\ref{thm:arb_rv-rep} and set
\[
Y_t^n = \int_{t_{n-1}}^{t}
\phi_{t_n}(s)\ud X_s,
\]
where $\phi_{t_n}(s)$ is the process constructed in
Lemma~\ref{lma:aux} such that $Y_t^n\rightarrow\infty$ as $t\rightarrow t_n$,
\[
\tau_n = \inf \bigl\{t\geq t_{n-1} \dvt
Y_t^n= |Z_{t_{n-1}}-Y_{t_{n-1}}| \bigr\},
\]
and for $s\in[t_{n-1},t_n)$,
\[
\Psi_T(s) = \phi_{t_n}(s) \operatorname{sign}(Z_{t_{n-1}}-Y_{t_{n-1}})
\mathbf{1}_{[0,\tau_n]}(s).
\]
Then $Y_{t_n} = Z_{t_{n-1}}$.

\item
Next, note that for a fixed $n$, the only possibility that $Y_{t_n}
\neq Z_{t_{n-1}}$ is that we have case A and $\tau_n \geq t_n$. Hence,
it suffices to show that the event
\[
C_n = \Bigl\{\sup_{t_{n-1}\leq t\leq t_n}n^{\kappa}|X_t
- X_{t_{n-1}}|\leq|Z_{t_{n-1}}-Z_{t_{n-2}}| \Bigr\}
\]
happens only finite number of times. For this we take $b\in
(\alpha- \frac{\kappa}{\gamma}, a )$, and the arguments in
\cite{m-s-v} implies that it is sufficient to show that only finite
number of events
\[
D_n = \Bigl\{\sup_{t_{n-1}\leq t\leq t_n}n^{\kappa}|X_t
- X_{t_{n-1}}|\leq\Delta_n^b \Bigr\}
\]
happen. Recall that now we have $b>\alpha-\frac{\kappa}{\gamma}$
which can be written as $\gamma b +\kappa> \gamma\alpha$. Hence we
can apply the small ball estimate (\ref{smallball}) of
Assumption~\ref{assu:smallball} together with Borel--Cantelli lemma to obtain the result.

\item
To complete the proof, we have to show that $\Vert \Psi_T\Vert _{T,\beta} <
\infty$ a.s. For this, we go through the main steps which are
different from the case of fBm. We write
\[
A_n = \bigl\{\mbox{We have Case A on }(t_{n-1},t_n]
\bigr\},\qquad B_n = A_n^C,
\]
and
\begin{eqnarray*}
\Psi_T(s) &=& \sum_{n\geq2}
\Psi_T(s)\mathbf{1}_{(t_{n-1},t_n]}(s)\mathbf{1}_{A_n}
\\
&&{}+ \sum_{n\geq2}\Psi_T(s)
\mathbf{1}_{(t_{n-1},t_n]}(s)\mathbf{1}_{B_n}
\\
&=:& \Psi_T^A(s) + \Psi_T^B(s).
\end{eqnarray*}
As for fBm case, it is evident that $\Vert \Psi_T^B(s)\Vert _{T,\beta}<\infty$
since only finite numbers of events $B_n$ happen. Furthermore, we can write
\begin{eqnarray*}
\E\bigl[\bigl \Vert \Psi_T^A(s)\bigr \Vert _{T,\beta}\bigr]
&=& \int_0^T \frac{\E|\Psi
_T^A(s)|}{s^{\beta}}
\\
&&{}+ \sum_{n=2}^{\infty} \int
_{t_{n-1}}^{t_n} \int_0^{t_{n-1}}
\frac
{\E|\Psi_T^A(t) - \Psi_T^A(s)|}{(t-s)^{\beta+1}}\ud s \ud t
\\
&&{}+\sum_{n=2}^{\infty} \int
_{t_{n-1}}^{t_n} \int_{t_{n-1}}^t
\frac
{\E|\Psi_T^A(t) - \Psi_T^A(s)|}{(t-s)^{\beta+1}}\ud s \ud t
\\
&=:& I_1+ I_2 + I_3.
\end{eqnarray*}
The finiteness of $I_1$ and $I_2$ are easy to show and we omit the
details. For $I_3$ we set $\lambda_n(t)=\operatorname{sign}(X_t -
X_{t_{n-1}})$ and obtain
\begin{eqnarray*}
I_3 &=& \sum_{n=2}^{\infty} \int
_{t_{n-1}}^{t_n} \int_{t_{n-1}}^t
\frac{\E|\Psi_T^A(t) - \Psi_T^A(s)|}{(t-s)^{\beta
+1}}\ud s \ud t
\\
&=& \sum_{n=2}^{\infty} n^{\kappa}\int
_{t_{n-1}}^{t_n} \int_{t_{n-1}}^t
\frac{\E|\lambda_n(t)\mathbf{1}_{s\leq\tau_n}-\lambda
_n(s)\mathbf{1}_{s\leq\tau_n}|\mathbf{1}_{A_n}}{(t-s)^{\beta
+1}}\ud s\ud t
\\
&\leq&\sum_{n=2}^{\infty} n^{\kappa}\int
_{t_{n-1}}^{t_n} \int_{t_{n-1}}^t
\frac{\E [|\lambda_n(t)-\lambda_n(s)|+\mathbf
{1}_{s\leq\tau_n< t} ]}{(t-s)^{\beta+1}}\ud s\ud t.
\end{eqnarray*}
Now note that
\[
\bigl |\lambda_n(t)-\lambda_n(s)\bigr | = \mathbf{1}_{\{X_s-X_{t_{n-1}}\leq0
\leq X_t-X_{t_{n-1}}\}}
+ \mathbf{1}_{\{X_s-X_{t_{n-1}}\geq0 \geq
X_t-X_{t_{n-1}}\}},
\]
and by taking expectation together with symmetry it is sufficient to
consider probability
\[
\P(X_s-X_{t_{n-1}}\leq0 \leq X_t-X_{t_{n-1}}).
\]
Let us study the integral
\[
\int_{t_{n-1}}^{t_n} \int_{t_{n-1}}^t
\frac{\P(X_s-X_{t_{n-1}}\leq0
\leq X_t-X_{t_{n-1}})}{(t-s)^{\beta+1}}\ud s\ud t.
\]
By change of variable, we obtain that it is sufficient to study
\begin{eqnarray*}
&&\int_0^{t_n-t_{n-1}} \int_0^t
\frac{\P
(X_{s+t_{n-1}}-X_{t_{n-1}}\leq0 \leq
X_{t+t_{n-1}}-X_{t_{n-1}})}{(t-s)^{\beta+1}}\ud s\ud t
\\
&&\quad=\int_0^{t_n-t_{n-1}} \int_0^{\fraca{t}{2}}
\frac{\P
(X_{s+t_{n-1}}-X_{t_{n-1}}\leq0 \leq
X_{t+t_{n-1}}-X_{t_{n-1}})}{(t-s)^{\beta+1}}\ud s\ud t
\\
&&\qquad{}+\int_0^{t_n-t_{n-1}} \int
_{\fraca{t}{2}}^t\frac{\P
(X_{s+t_{n-1}}-X_{t_{n-1}}\leq0 \leq
X_{t+t_{n-1}}-X_{t_{n-1}})}{(t-s)^{\beta+1}}\ud s\ud t
\\
&&\quad=: J_1 + J_2.
\end{eqnarray*}
For $J_1$ we can bound the probability with one and get
\[
J_1 \leq C \Delta_n^{1-\beta}.
\]
Consider next the term $J_2$. By assumption 1 of
Definition~\ref{defn:Xalpha} the covariance of Gaussian processes
$X_{s+t_{n-1}}-X_{t_{n-1}}$ and $X_{t+t_{n-1}}-X_{t_{n-1}}$ is positive
for every $n$ and every $s,t\in[0,t_n-t_{n-1}]$.
Thus we can apply Corollary~\ref{lma:crossing} and assumption 4 to obtain
\[
\P(X_s-X_{t_{n-1}}\leq0 \leq X_t-X_{t_{n-1}})
\leq C\frac{\sqrt
{W_n(t,s)}}{\sqrt{\E(X_s - X_{t_{n-1}})^2}},
\]
where
\begin{eqnarray*}
W_n(t,s) &=& \E\bigl(X_{t+t_{n-1}} -X_{t_{n-1}} -
(X_{s+t_{n-1}} - X_{t_{n-1}})\bigr)^2
\\
&\leq& C(t-s)^{2\alpha},
\end{eqnarray*}
and
\[
\E(X_{s+t_{n-1}} - X_{t_{n-1}})^2 \geq Cs^2
\]
by assumptions. Hence, by symmetry of probabilities
$P(X_s-X_{t_{n-1}}\leq0 \leq X_t-X_{t_{n-1}})$ and
$P(X_s-X_{t_{n-1}}\geq0 \geq X_t-X_{t_{n-1}})$, we obtain
\begin{eqnarray*}
J_2 &\leq& C\int_0^{t_n-t_{n-1}} \int
_{\fraca{t}{2}}^t\frac
{(t-s)^{\alpha-\beta-1}}{s}\ud s\ud t
\\
&\leq& C\int_0^{t_n-t_{n-1}} t^{\alpha-\beta-1}\ud t
\\
&\leq& C \Delta_n^{\alpha-\beta}.
\end{eqnarray*}
To conclude, we note that
\[
\int_{t_{n-1}}^{t_n} \int_{t_{n-1}}^t
\frac{\mathbf{1}_{s\leq\tau
_n < t}}{(t-s)^{\beta+1}}\ud s \ud t \leq C\Delta_n^{1-\beta},
\]
and hence
\[
I_3 \leq C \sum_{n=2}^{\infty}
n^{\kappa- \gamma(\alpha-\beta)} < \infty
\]
by our choice of $\kappa$, $\gamma$, and $\beta$.\qed
\end{enumerate}
\noqed\end{pf*}
%
\begin{rmk}\label{rmk:H}
With our general assumptions, we can only cover H\"{o}lder continuous
variables of order $a>1-\alpha$.
However, under additional assumption that for $s$ close to $T$ and
small enough $\Delta$ the incremental variance satisfies
\[
\E[X_{s+\Delta} - X_s]^2 \geq C
\Delta^{2\theta}
\]
with some constant $C$ and some parameter $\theta\in(\alpha,1)$, we
can cover more. More precisely, we can cover H\"{o}lder continuous
processes of order $a>\theta- \alpha$. Especially this is the case if
the process $X$ is stationary or has stationary increments with
$W_X(0,t)\sim t^{2\alpha}$. In particular case of fBm one can cover
H\"{o}lder continuous processes of any order $a>0$. Similarly, with a
process $\tilde{X}$ one can cover H\"older continuous processes of any
order $a>0$.
\end{rmk}
%
\begin{rmk}
In \cite{m-s-v}, the authors proved also that under additional
assumption that $\Psi$ is continuous, the assumption of the Theorem~\ref{thm:H_rv-rep} is also necessary. The proof is based only to the
H\"older continuity of fBm and well-known properties of Young
integrals. Consequently, same conclusion remains for our general class
of processes.
\end{rmk}

\begin{cor}
Let $Z_t$ be a.s. H\"{o}lder continuous process of order $a>1-\alpha$
and for every $t\in(0,T]$ we have $X\in\mathcal{X}^\alpha_t$. Then
for every $t$ there
exists $\mathbb{F}$-adapted process $\Psi_t$ such that it holds, a.s.,
\[
\int_0^t \Psi_t(s)\ud
X_s = Z_t,
\]
i.e. the integral $\int_0^t \Psi_t(s)\ud X_s$ is a version of $Z_t$.
\end{cor}

\section{Applications and discussions}
\label{sec:app}

In the paper \cite{m-s-v}, the authors considered financial
implications of their results to a model where the stock is driven by
geometric fBm. In particular, the results indicate one more reason why
geometric fBm is not a proper model in finance.
Evidently, we could state similar results in our general setting and as
a consequence, we can argue that processes $X\in\mathcal{X}^\alpha
_T$ do not fit
well as the driving process of stock prices. This is also discussed
with details in \cite{s-v2} where the authors proved the pathwise It\^
o--Tanaka formula for processes in our class. For further details, we
refer to \cite{m-s-v} and \cite{s-v2}, and the repetition of the
arguments presented in \cite{m-s-v} for more general processes $X\in
\mathcal{X}^\alpha_T$ are left to the reader. However, we wish to
give one remark on financial implications of our results. In \cite{m-s-v}, the authors proved that if the stock is driven by geometric
fractional Brownian motion, then one can replicate essentially all
interesting derivatives. On the other hand, we can never know whether
the process driving the stock is geometric fBm or not. The benefit of
our results is that in addition to the fact that the replication can be
done with much more general class of processes, the replication can be
done also in arbitrary small amount of time. This means that one can
wait and observe the process up to some time arbitrary close to the
maturity, and start the replication procedure after that point.
Especially this is useful if there is no information on the stock
dynamics. Assuming that the driving process is Gaussian, one can save
time to estimate the covariance structure of the process and use this
information for the replication.

\subsection*{On the uniqueness of representation}

In the case of standard Brownian motion, every centered random variable
$\xi$ with finite variance can be represented as
\[
\xi= \int_0^1 \Psi(s)\ud W_s,
\]
where $\int_0^1 \E[\Psi(s)]^2\ud s<\infty$.
Moreover, a direct consequence of the It\^o isometry implies that in
this case the process $\Psi$ is unique. However, for generalised
Lebesgue--Stieltjes integrals the representation is not unique. As an
example, consider fractional Ornstein--Uhlenbeck process given by
\[
U_t^\theta= \int_0^t
\mathrm{e}^{-\theta(t-s)}\ud B^H_s.
\]
On the other hand, by Theorem~\ref{thm:H_rv-rep} we know that
\[
U_t^\theta= \int_0^t
\Psi_t(s)\ud B^H_s,
\]
where $\Psi_t(s)$ is defined equally zero on interval $[0,t_1]$, and
$t_1$ can be chosen arbitrary close to $1$. Hence, the representation
is clearly not unique in general with pathwise integrals. On the other
hand, for Skorokhod integrals with respect to fBm the representation is
unique (see \cite{bender}).

\subsection*{The problem of zero integral}

Another application which was considered in \cite{m-s-v} for fBm was
the problem of zero integral, and we wish to end the paper by giving
some remarks on zero integral problem for our general class of processes.

Recall that the zero integral problem refers to the question whether we
have implication
%
\begin{equation}
\label{zero_integral_conc} \int_0^1 u_s \ud
X_s = 0, \qquad\mbox{a.s.}\quad\Rightarrow\quad u_s = 0,
\qquad \P\otimes\operatorname{Leb}\bigl([0,T]\bigr) \qquad\mbox{a.e.}
\end{equation}
For standard Brownian motion this is true under assumption $\int_0^T
\E[u_s^2]\ud s < \infty$, and the result is a direct consequence of
the It\^o isometry. On the other hand, if we only have that
$\int_0^1 u_s^2 \ud s < \infty$ a.s., then the conclusion is false.
In particular, one can construct an adapted process such that $\int_0^{\fraca{1}{2}} u_s\ud W_s = 1$ and $\int_{\fraca{1}{2}}^1 u_s\ud
W_s = -1$.

Similarly for fBm, the authors in \cite{m-s-v} explained that one can
construct an adapted process such that $\int_0^{\fraca{1}{2}} u_s\ud
B^H_s = 1$ and $\int_{\fraca{1}{2}}^1 u_s\ud B^H_s = -1$. Now the
results presented in this paper indicate that the same conclusion
remains true if we replace fBm $B^H$ with more general Gaussian process
$X$. This suggests that the problem of zero integral is not interesting
in the first place since the conclusion is false in most of the
interesting case unless one poses some extra assumptions. We also note
that a negative answer to the question of zero integral is a direct
consequence of the fact that the representation is not unique. As
another example of this, consider a random variable $(X_1 - K)^+$.
Clearly this random variable is an end value of H\"{o}lder continuous
process, and thus Theorem~\ref{thm:H_rv-rep} implies that there is a
process $\Psi_1(s)$ such that
\[
(X_1 - K)^+ = \int_0^1
\Psi_1(s)\ud X_s.
\]
Moreover, by construction of the process $\Psi_1(s)$ we have
$\Psi_1(s) = 0$ on the interval $s\in[0,t_1]$. On the other hand, by
Theorem~\ref{thm:H_rv-rep} (assuming that the covariance $R_X$ of the
process $X$ itself satisfies 1--4) we have
\[
(X_1 - K)^+ = (X_0 - K)^+ +\int_0^1
\mathbf{1}_{X_s>K}\ud X_s.
\]
If now $X_0\leq K$ a.s., subtracting first equation from the second
one, we obtain that
\[
0 = \int_0^1 \Psi_1(s) -
\mathbf{1}_{X_s > K}\ud X_s.
\]
Now $\Psi_1(s)=0$ a.s. on $[0,t_1]$, and clearly the same is not true for
process $\mathbf{1}_{X_s>K}$. This is another argument to show that
the $\int_0^1 u_s\ud X_s=0$ does not imply $u_s=0$ a.s. in general.


\section*{Acknowledgement}

Lauri Viitasaari thanks The Finnish Graduate School in Stochastic and
Statistics for financial support.


%

\printhistory
\end{document}